\documentclass[a4paper,10pt,reqno]{amsart}

\usepackage[UKenglish]{babel}

\usepackage[utf8]{inputenc}
\DeclareUnicodeCharacter{2009}{FIXME}

\usepackage{amssymb}
\usepackage{amsmath}
\usepackage{amsthm}
\usepackage{bbm}
\usepackage{verbatim, stackrel}
\usepackage{comment,cite}	
\usepackage{bigints}

\usepackage{hyperref}
\hypersetup{
    colorlinks=true,
    linkcolor=blue,
    citecolor=red,
    filecolor=magenta,      
    urlcolor=cyan,
    pdftitle={A characterization of the individual maximum and anti-maximum principle},
    bookmarks=true,
    linktocpage=true,
}

\usepackage[shortlabels]{enumitem}
\usepackage{marginnote}
\usepackage{color}

\newcommand{\bbC}{\mathbb{C}}

\newcommand{\bbN}{\mathbb{N}}

\newcommand{\bbR}{\mathbb{R}}

\newcommand{\calL}{\mathcal{L}}

\newcommand{\calT}{\mathcal{T}}

\DeclareMathOperator{\id}{id} 
\DeclareMathOperator{\one}{\mathbbm{1}} 
\newcommand{\argument}{\mathord{\,\cdot\,}} 
\newcommand{\dx}{\;\mathrm{d}} 
\newcommand{\norm}[1]{\left\lVert #1 \right\rVert} 
\newcommand{\modulus}[1]{\left\lvert #1 \right\rvert} 
\DeclareMathOperator{\dom}{dom} 
\DeclareMathOperator{\trace}{Tr} 
\newcommand{\restrict}[1]{|_{#1}}

\newcommand{\spb}{s} 
\newcommand{\Res}{\mathcal{R}} 

\theoremstyle{definition}
\newtheorem{definition}{Definition}[section]
\newtheorem{remark}[definition]{Remark}
\newtheorem{remarks}[definition]{Remarks}

\newtheorem{setting}[definition]{Setting}

\theoremstyle{plain}
\newtheorem{proposition}[definition]{Proposition}
\newtheorem{lemma}[definition]{Lemma}
\newtheorem{theorem}[definition]{Theorem}
\newtheorem{corollary}[definition]{Corollary}

\numberwithin{equation}{section}

\begin{document}

\title[Individual anti-maximum principles]{A characterization of the individual maximum and anti-maximum principle}
\author{Sahiba Arora}
\address{Sahiba Arora, Technische Universität Dresden, Institut für Analysis, Fakultät für Mathematik , 01062 Dresden, Germany}
\email{sahiba.arora@mailbox.tu-dresden.de}
\author{Jochen Gl\"uck}
\address{Jochen Gl\"uck, Bergische Universität Wuppertal, Fakultät für Mathematik und Naturwissenschaften, 42119 Wuppertal, Germany}
\email{glueck@uni-wuppertal.de}

\date{\today}
\begin{abstract}
	Abstract approaches to maximum and anti-maximum principles for differential operators typically rely on the condition that all vectors in the domain of the operator are dominated by the leading eigenfunction of the operator.
	We study the necessity of this condition. 
	In particular, we show that under a number of natural assumptions, so-called individual versions of both the maximum and the anti-maximum principle simultaneously hold if and only if the aforementioned domination condition is satisfied.
	
	Consequently, we are able to show that a variety of concrete differential operators do not satisfy an anti-maximum principle. 
\end{abstract}

\maketitle

\section{Introduction} \label{section:introduction}

Maximum and anti-maximum principles have been long used to obtain information about PDEs without explicitly knowing their solutions. 
While (anti-)maximum principles have been proved/disproved, for various concrete differential operators -- using, for  instance, kernel estimates -- only a few endeavours have been made to integrate these arguments in an abstract setting. 
One attempt in this direction was made by Takáč in \cite{Takac1996}. Inspired by his techniques, the present authors recently proved necessary and sufficient conditions for the so-called \emph{uniform} (anti-)maximum principles in \cite{AroraGlueck2021b}. On the other hand, a characterization for the \emph{individual} maximum and anti-maximum principle was given in \cite[Theorem~4.4]{DanersGlueckKennedy2016b}. 

More precisely, let $\Omega$ be a bounded domain in $\bbR^d$ and let $A: E\supseteq \dom{A} \to E$ be a differential operator on a function space $E$ over $\Omega$. Consider the equation
\[
	(\lambda-A)u=f
\]
for real numbers $\lambda$ in the resolvent set of $A$. If $\lambda_0 \in \bbR$ is an isolated spectral value of $A$, then we say that $A$ satisfies the \emph{individual} maximum principle if the inequality $f\geq 0$ implies $u\geq 0$ for all $\lambda$ in a $f$-dependent right neighbourhood of $\lambda_0$. If the right neighbourhood can be chosen independently of $f$, then we call the maximum principle \emph{uniform}. Analogously, if $f\geq 0$ implies $u\leq 0$ in a left neighbourhood of $\lambda_0$, then we speak of (individual or uniform) anti-maximum principles.

\subsection*{Eventual positivity and a domination assumption}

Recall that $A$ is said to be \emph{resolvent positive} if the resolvent $\Res(\lambda,A)$ exists and is positive for all sufficiently large real numbers $\lambda$. Resolvent positivity was studied in detail in \cite{Arendt1987} and is also related to the theory of \emph{positive $C_0$-semigroups} (see \cite[Chapters~B-II and C-II]{Nagel1986} and \cite[Corollary~11.4]{BatkaiKramarRhandi2017}). This notion also appears in \cite{Takac1996}.

As a generalization of the above notion, \emph{eventual positivity} and \emph{eventual negativity} of the resolvent was studied, as a complement to \emph{eventually positive $C_0$-semigroups}, in \cite{DanersGlueckKennedy2016a, DanersGlueckKennedy2016b, DanersGlueck2017, DanersGlueck2018a, DanersGlueck2018b}. The notions of individual and uniform eventual positivity of the resolvent, actually coincide with the individual and uniform maximum principles discussed above. The same is true for the anti-maximum principle and the eventual negativity of the resolvent.

In the present article, we deal with the individual (anti-)maximum principles and, in particular, with a domination condition that occurred in the characterization theorem \cite[Theorem~4.4]{DanersGlueckKennedy2016b}. To recall this condition, let $E$ be a complex Banach lattice whose real part and positive cone will be denoted by $E_{\bbR}$ and $E_+$ respectively. For $0\lneq u\in E$, the set
\[
	E_u:=\{f\in E: \text{ there exists } c>0 \text{ such that } \modulus{f} \leq cu \}
\]
is called the \emph{principal ideal} generated by $u$. Equipped with the \emph{gauge norm}
\[
	\norm{f}_u := \inf \{c>0 : \modulus {f} \leq cu\}\qquad \text{ for } f \in E_u,
\]
it is also a complex Banach lattice. The canonical embedding $E_u\to E$ is continuous and a lattice homomorphism. If $E_u$ is dense in $E$, then we say that $u$ is a \emph{quasi-interior point} of $E$ (or more precisely of $E_+$). As an example, let $(\Omega,\mu)$ be a finite measure space and $p\in [1,\infty]$. Then $\one$ -- which stands for the constant function taking the value $1$ -- is a quasi-interior point of the Banach lattice $L^p(\Omega,\mu)$ and the corresponding principal ideal is given by $L^p(\Omega,\mu)_{\one}=L^\infty(\Omega,\mu)$ (here the gauge norm is simply the $\norm{\argument}_{\infty}$-norm). We refer to the standard monographs \cite{Schaefer1974} and \cite{Meyer-Nieberg1991} for the general theory of Banach lattices.

For a linear operator $A: E \supseteq \dom{A} \to E$, the condition
\[
	\dom{A}\subseteq E_u
\]
is referred to as the domination condition in \cite{DanersGlueckKennedy2016b, DanersGlueck2017} and plays a significant role in the characterization of the individual (anti-)maximum principle in \cite[Theorem~4.4]{DanersGlueckKennedy2016b}.  In \cite[Theorem~4.1]{DanersGlueck2017}, the condition was analysed in detail and it was shown that the necessary conditions in \cite[Theorem~4.4]{DanersGlueckKennedy2016b} remain true without it. In this paper, we show that under a spectral assumption and a much weaker domination assumption (see Setting~\ref{sett:main} below), the individual maximum and anti-maximum principles are concomitantly satisfied if and only if $\dom{A}\subseteq E_u$. This helps us in disproving anti-maximum principles for various differential operators for which the maximum principle is already known.

\subsection*{Main result}

We confine ourselves to the study of real linear operators. This is not too restrictive, for instance, differential operators with real coefficients are typically real. Concretely, we call a linear operator $A: E \supseteq \dom{A} \to E$ \emph{real} if $\dom{A} = \dom{A}\cap E_{\bbR} + i \dom{A}\cap E_{\bbR}$ and $A\big(\dom{A}\cap E_{\bbR}\big)\subseteq E_{\bbR}$. Observe that if $A$ is real, then so are the operators $\Res(\lambda,A)$ for all real numbers $\lambda$ in the resolvent set $\rho(A)$.
For two vectors $u,v\in E_{\bbR}$, we write $u \succeq v$ (equivalently, $v\preceq u$) if there exists a real number $c > 0$ such that $u\geq cv$. 
We use the same notation not only for vectors but also for operators.

Most of our results are formulated in the following setting:

\begin{setting}
	\label{sett:main}
	Let $A: E \supseteq \dom{A} \to E$ be a closed, densely defined, and real linear operator 
	on a complex Banach lattice $E$ and let $\lambda_0 \in \bbR$ be an isolated spectral value
	of $A$
	and a pole of the resolvent $\Res(\argument,A)$. 
	Moreover, fix a quasi-interior point $u \in E$.
	\begin{itemize}
		\item 
		We say that the \emph{spectral assumption} is satisfied if and only if the following holds:
		the eigenvalue $\lambda_0$ of $A$ is geometrically simple and the corresponding eigenspace
		$\ker(\lambda_0-A)$ is spanned by a vector $v$ which satisfies $v\succeq u$;
		moreover, the dual eigenspace $\ker(\lambda_0 - A')$ contains a strictly positive functional $\varphi$.
		
		\item 
		We say that the \emph{domination assumption} is satisfied if and only if 
		there exists an integer $n \ge 0$ such that $\dom{A^n} \subseteq E_u$.
	\end{itemize}
\end{setting}

Observe that if the domination assumption above holds, then the condition ``$u$  is a quasi-interior point of $E$'' is redundant because $A$ (and hence $A^n$) is a densely defined operator.

The following is our main result.

\begin{theorem}
	\label{thm:main}
	In Setting~\ref{sett:main}, let the domination assumption be satisfied. 
	Then the following assertions are equivalent.
	\begin{enumerate}[\upshape (i)]
		\item 
		\emph{Strong individual (anti-)maximum principle:}
		For every $0 \lneq f \in E$, we have 
		\begin{align*}
			\Res(\mu,A)f \preceq -u
			\qquad \text{and} \qquad 
			\Res(\lambda,A)f \succeq u
		\end{align*}
		for all $\mu$ in a $f$-dependent left neighbourhood of $\lambda_0$ and
		for all $\lambda$ in a $f$-dependent right neighbourhood of $\lambda_0$.
		
		\item 
		\emph{Individual (anti-)maximum principle plus spectral assumption:}
		The spectral assumption is satisfied and 
		for every $0 \leq f \in E$, we have 
		\begin{align*}
			\Res(\mu,A)f \le 0
			\qquad \text{and} \qquad 
			\Res(\lambda,A)f \ge 0
		\end{align*}
		for all $\mu$ in a $f$-dependent left neighbourhood of $\lambda_0$ and
		for all $\lambda$ in a $f$-dependent right neighbourhood of $\lambda_0$.
		
		\item 
		\emph{Improved domination assumption plus spectral assumption:}
		The spectral assumption is satisfied and $\dom{A} \subseteq E_u$.
	\end{enumerate}
\end{theorem}

Let us briefly discuss the particular case where $E=C(K)$, the space of continuous functions on a compact Hausdorff set $K$ and $u\in E$ is the constant one function. Since $E_u=E$, the domination condition is always satisfied in Setting~\ref{sett:main} for $n=1$. Thus, according to Theorem~\ref{thm:main}, both strong individual maximum and anti-maximum principles hold at $\lambda_0$ if and only if the spectral assumption is satisfied. 

Next, let $E=L^p(\Omega,\mu)$, where $(\Omega,\mu)$ is a finite measure space and $p\in[1,\infty)$. If $u$ is the constant one function, then as mentioned before, $E_u=L^\infty(\Omega,\mu)$. For numerous differential operators $A$, one can show using Sobolev embedding theorems that the domination condition holds for sufficiently large $n\in\bbN$ but not for $n=1$. As a result, if one knows a priori that the maximum principle holds at $\lambda_0$ for those operators that in addition satisfy the spectral assumption, then the anti-maximum principle cannot hold at $\lambda_0$. Analogously, if the anti-maximum principle holds at $\lambda_0$, then the maximum principle cannot. We will illustrate this in Section~\ref{section:applications}.

It is natural to ask whether the assertions in Theorem~\ref{thm:main} are equivalent without the domination assumption in Setting~\ref{sett:main}. However, one can immediately check that this is not the case. Indeed, on $L^2[0,1]$  consider the bounded operator defined by $Af= (\one \otimes \one)f-f$ for all $f\in L^2[0,1]$, and let $\lambda_0 = 0$. The resolvent can be computed directly and is given by
\[
	\Res(\lambda,A)f =\frac{1}{\lambda(\lambda+1)} \big( (\one \otimes \one)f+\lambda f \big),\qquad (\lambda\in \rho(A)).
\]
Thus, one has $\Res(\lambda,A) \succeq \one\otimes \one$ and $\Res(\mu,A)\preceq -\one\otimes \one$ for all $\lambda>0$ and for all $\mu\in (-1,0)$. However, $\dom(A^n) = L^2[0,1]\not\subseteq L^\infty[0,1]=L^2[0,1]_{\one}$ for any $n\in\bbN_0$.

\subsection*{Related literature}

Individual maximum and anti-maximum principles for concrete differential operators have been intensively studied for several decades.  For instance, we refer the reader to the articles \cite{ProtterWeinberger1999, Pinchover1999, GrunauSweers2001, ClementPeletier1979, Birindelli1995} and to the references in Section~\ref{section:applications} for various examples. Abstract approaches to (anti-)maximum principles with several applications have been taken \cite{DanersGlueckKennedy2016a, DanersGlueckKennedy2016b, Takac1996, AroraGlueck2021b}. Recently, the first author proved \emph{local} versions of the (anti-)maximum principles in \cite{Arora2022}, see also \cite[Chapter~9]{Arora2023}, and gave a number of applications.

\subsection*{Organization of the article}

In Section~\ref{section:operator-ranges}, we recall the notion of operator ranges, a concept which is instrumental in proving Theorem~\ref{thm:main} in Section~\ref{section:proof}. In Section~\ref{section:sufficient-condition-domination}, we discuss a characterization of the domination assumption in Setting~\ref{sett:main} in the important special case where $A$ generates a $C_0$-semigroup. Applications of our result to various concrete differential operators are presented in Section~\ref{section:applications}.

\section{Eventual conditions on operator ranges}
	\label{section:operator-ranges}

The improved domination condition $\dom A \subseteq E_u$ from Theorem~\ref{thm:main} can equivalently be written in the form $\Res(\lambda,A)E \subseteq E_u$ for one number $\lambda$ in the resolvent set of $A$ or, equivalently, for all $\lambda$ in the resolvent set.
This raises the following question:~what happens if we only know that $\Res(\lambda_f,A)f \in E_u$ for all $f \in E$ and certain $f$-dependent numbers $\lambda_f$? 
The purpose of this section is to analyse this situation; see Corollary~\ref{cor:eventually-in-subspace-countable}(a) below.
In fact, this will be a consequence of the following more general result:

\begin{theorem}
	\label{thm:countable-individual-range-condition}
	Let $E$, $F$, and $V$ be Banach spaces such that $V$ is continuously embedded in $F$.
	Let $\calT \subseteq \calL(E;F)$ be a countable set of bounded linear operators and assume that for each $f \in E$ there exists $T \in \calT$ such that $Tf \in V$. 
	Then there even exists an operator $T_0 \in \calT$ such that $T_0 E \subseteq V$.
\end{theorem}

We will show below that this result is a consequence of Baire's category theorem. But first, let us list a few consequences:

\begin{corollary}
	\label{cor:eventually-in-subspace-countable}
	Let $E$ and $V$ be two complex Banach spaces such that $V$ is continuously embedded in $E$, and let $A: E \supseteq \dom A \to E$ be a linear operator.
	\begin{enumerate}[\upshape (a)]
		\item
		Let $C \subseteq \bbC$ be a countable subset of the resolvent set of $A$ and assume that for each $f \in E$ there exists a number $\lambda \in C$ such that $\Res(\lambda,A)f \in V$.
		Then $\dom A \subseteq V$.
		
		\item 
		Assume that $A$ generates a $C_0$-semigroup $(e^{tA})_{t \in [0,\infty}$ on $E$ and assume that there exists a countable set $S \subseteq [0,\infty)$ such that for each $f \in E$ there exists a time $t \in S$ for which we have $e^{tA}f \in V$.
		Then there exists a time $t_0 \in [0,\infty)$ such that $e^{tA}E \subseteq V$ for each $t \ge t_0$.
	\end{enumerate}
\end{corollary}
\begin{proof}
	(a) 
	From the assumption and Theorem~\ref{thm:countable-individual-range-condition} we conclude that, for some $\lambda \in C$, $\dom{A} = \Res(\lambda,A)E \subseteq V$.
	
	(b)
	It follows from the assumption and from Theorem~\ref{thm:countable-individual-range-condition} that, for some $t_0 \in S$, we have $e^{t_0A}E \subseteq V$. By the semigroup law, the same is also true for all times $t \ge t_0$. 
\end{proof}

We will use assertion~(a) of the corollary in the next section in the proof of our main theorem. 

\begin{remarks}
	(a) Corollary~\ref{cor:eventually-in-subspace-countable}(b) is related to the theory of eventually positive semigroups:
	If $(e^{tA})_{t \in [0,\infty)}$ is a $C_0$-semigroup on a Banach lattice $E$, and $u \in E_+$ is a quasi-interior point, an assumption that occurs frequently in the context of eventually positive semigroups is the existence of a time $t_0$ such that $e^{t_0 A}E \subseteq E_u$; see, for instance, \cite[Section~5]{DanersGlueckKennedy2016a}.
	Corollary~\ref{cor:eventually-in-subspace-countable}(b) shows that this is equivalent to the a priori weaker assumption that each orbit of the semigroup is eventually in $E_u$; here we have tacitly used the semigroup law.
	
	(b) On the other hand, the theory of eventually positive semigroups also shows that a similar result as in Corollary~\ref{cor:eventually-in-subspace-countable}(b) does not hold if we replace subspaces with cones:
	In \cite[Examples~5.7 and~5.8]{DanersGlueckKennedy2016a}, one can find an example of a semigroup $(e^{tA})_{t \in [0,\infty)}$ on a Banach lattice $E$ such that, for each $f \in E_+$, the orbit of $f$ is eventually in $E_+$ (thus, we can take the countable subset $S$ to be $\bbN$), but one does not have $e^{tA}E_+ \subseteq E_+$ for any $t \in (0,\infty)$.
	
	(c)
	In a similar way and again with the aid of \cite[Example~5.7]{DanersGlueckKennedy2016a}, one can see that Corollary~\ref{cor:eventually-in-subspace-countable}(a) need not hold if $V$ is merely a cone but not a subspace.
	
	(d) 
	By applying Corollary~\ref{cor:eventually-in-subspace-countable}(b) to $V = \dom A$, we can see that a $C_0$-semigroup is eventually differentiable (in the sense of \cite[Definition~II.4.13]{EngelNagel2000}) if each of its individual orbits is eventually differentiable. 
	It turns out, though, that the latter result even holds for much more general classes of operator-valued functions than only $C_0$-semigroups; this was recently proved by Peruzzetto in \cite[Proposition~3.11]{Peruzzetto2022} (also based on Baire's category theorem, but with a different technical set-up).
\end{remarks}

To derive Theorem~\ref{thm:countable-individual-range-condition} from Baire's theorem, we need the concept of \emph{operator ranges}.
A subspace $V$ of a Banach space $F$ is called an \emph{operator range} if there exists a Banach space $E$ and a bounded linear operator $T:E\to F$ such that $T(E)=V$. For more details about operator ranges, we refer the reader to \cite{Cross1980, CrossOstrovskiiShevchik1995}.  Recent results about the eventual invariance of operator ranges were given by the present authors in \cite[Section~2]{AroraGlueck2023}.
Here, we only recall some of their basic properties.
The proof of the following proposition can be found in \cite[Proposition~2.1]{Cross1980}:

\begin{proposition}
	\label{prop:characterization-operator-range}
	Let $F$ be a Banach space and let $V \subseteq F$ be a vector subspace of $F$. The following are equivalent.
	\begin{enumerate}[\upshape (i)]
		\item The subspace $V$ is an operator range.
		\item There exists a closed operator on $F$ whose domain is $V$.
		\item There exists a complete norm $\norm{\argument}_V$ on $V$ which makes the embedding $(V, \norm{\argument}_V) \hookrightarrow (F, \norm{\argument}_F)$ is continuous.
	\end{enumerate}
\end{proposition}

Obviously, the image of an operator range under a bounded linear operator is again an operator range.
A bit more interestingly, the same holds for pre-images:

\begin{proposition}
	\label{prop:preimage-operator-range}
	Let $S: E \to F$ be a bounded linear operator between Banach spaces $E$ and $F$.
	If a vector subspace $V \subseteq F$ is an operator range in $F$, then the pre-image $U := S^{-1}(V)$ is an operator range in $E$.
\end{proposition}

\begin{proof}
	Let $\norm{\argument}_V$ denote a complete norm on $V$ which makes the embedding into $F$ continuous (such a norm exists by Proposition~\ref{prop:characterization-operator-range}). 
	We define a norm $\norm{\argument}_U$ on $U = S^{-1}(V)$ by the formula
	\[
		\norm{x}_U:= \norm{x}_E + \norm{S x}_V
	\]
	for all $x\in U$. 
	It is not difficult to check that $(U, \norm{\argument}_U)$ is a Banach space, and obviously the inclusion $(U, \norm{\argument}_U) \hookrightarrow (E, \norm{\argument}_E)$ is continuous; thus, $U$ is an operator range.
\end{proof}

It is a classical observation in functional analysis that, as a consequence of Baire's theorem, a countable union of closed proper subspaces of a Banach space $F$ cannot be equal to $F$.
For the proof of Theorem~\ref{thm:countable-individual-range-condition}, we recall that the same is true not only for closed subspaces but even for operator ranges:

\begin{proposition}
	\label{prop:necessary-operator-range}
	Let $F$ be a Banach space and let $U\subseteq F$ be an operator range. If $U$ is a proper subspace (i.e., $U \not= F)$, then it is a meagre set in $F$.
	
	Consequently, if $F = \cup_{k\in\bbN} U_k$, where each $U_k\subseteq F$ is an operator range, then there exists $k\in\bbN$ such that $U_k=F$.
\end{proposition}

\begin{proof}
	Since $U$ is an operator range, there exists a Banach space $E$ and $T\in \calL(E,F)$ such that $T(E)=U$. If $U$ is not a meagre set, then by a version of open mapping theorem \cite[Theorem~2.11]{Rudin1991}, we have $U=F$. The second assertion follows now by Baire's theorem.
\end{proof}

We conclude this section with the proof of Theorem~\ref{thm:countable-individual-range-condition}.

\begin{proof}[Proof of Theorem~\ref{thm:countable-individual-range-condition}]
	Enumerate the elements of $\calT$ as $T_1, T_2, \ldots$ (where some operators might occur infinitely often if $\calT$ is finite). For each $k \in \bbN$, consider the subspace
	\begin{align*}
		U_k := T_k^{-1}(V)
	\end{align*}
	of $E$. 
	According to Propositions~\ref{prop:characterization-operator-range} and~\ref{prop:preimage-operator-range}, each of the spaces $U_k$ is an operator range in $E$, and the assumption of the theorem implies that $\cup_{k \in \bbN} U_k = E$.
	Hence, Proposition~\ref{prop:necessary-operator-range} shows that there exists an index $k_0$ for which we have $U_{k_0} = E$; thus, $T_{k_0}E \subseteq V$.
\end{proof}

\section{Proof of the main result}
	\label{section:proof}

The proof of Theorem~\ref{thm:main} is based on a couple of auxiliary results. 
We start with the following finite series expansion of resolvents.

\begin{lemma}
	\label{lem:resolvent-expansion-at-multiple-points}
	Let $A: E \supseteq \dom{A} \to E$ be a closed linear operator on a complex Banach space $E$,
	let $m \ge 0$ be an integer and let $\lambda, \mu_1, \ldots, \mu_m \in \bbC$ 
	be points in the resolvent set of $A$.
	Then the expansion formula
	\begin{align*}
		\Res(\lambda,A)
		= &
		\sum_{k=1}^{m} 
		\Big( \prod_{j=1}^{k-1} (\mu_j - \lambda) \Big) 
		\Big( \prod_{j=1}^k \Res(\mu_j,A) \Big) 
		\\
		&
		+
		\Res(\lambda,A)
		\Big( \prod_{j=1}^m (\mu_j - \lambda) \Big) 
		\Big( \prod_{j=1}^m \Res(\mu_j,A) \Big)
	\end{align*}
	holds.
\end{lemma}
\begin{proof}
	The formula follows from an iterative application of the resolvent identity:~first apply it to the numbers $\lambda,\mu_1$,
	then to the numbers $\lambda,\mu_2$, and proceed inductively.
\end{proof}

\begin{remark}
	\label{rem:dom-assumption}
	Suppose the domination assumption in Setting~\ref{sett:main} is satisfied and $\mu_1,\mu_2,\ldots,\mu_n$ 
	are real numbers in the resolvent set of $A$. Then by the closed graph theorem, $\prod_{j=1}^n \Res(\mu_j,A)$ is a bounded operator from $E$ to $E_u$. Hence for every $f\in E_{\bbR}$, we have $-u \preceq \prod_{j=1}^n \Res(\mu_j,A)f \preceq u$.
\end{remark}

By using the previous lemma, we now show how upper/lower estimates for the resolvent at a point
can be transferred to other points in the resolvent set.
Related results can be found in \cite[Section~4]{AroraGlueck2021b}, but the point of the following theorem is that the number $\mu$ is allowed to depend on $f$, which makes the situation more subtle.

\begin{theorem}
	\label{thm:half-estimates}
	In Setting~\ref{sett:main}, let both the domination and spectral assumption be satisfied. 
	\begin{enumerate}[\upshape (a)]
		\item 
		Assume that for each $0 \le f \in E$ there exists a number $\mu > \lambda_0$ 
		in the resolvent set of $A$ such that $\Res(\mu,A)f \succeq -u$.
		
		Then we have $\Res(\lambda,A)f \succeq -u$ for all $0 \le f \in E$ and all $\lambda < \lambda_0$ in the resolvent set of $A$.
		
		\item 
		Assume that for each $0 \le f \in E$ there exists a number $\mu < \lambda_0$ 
		in the resolvent set of $A$ such that $\Res(\mu,A)f \preceq u$.
		
		Then we have $\Res(\lambda,A)f \preceq u$ for all $0 \le f \in E$ and all $\lambda > \lambda_0$ in the resolvent set of $A$.
	\end{enumerate}
\end{theorem}

\begin{proof}
	It is sufficient to prove (a); assertion (b) will then follow by replacing $A$ with $-A$. To this end, let $0 \le f \in E$ and let $n \ge 0$ be the integer from the domination assumption (for which we have $\dom A^n \subseteq E_u$).
	We will recursively construct numbers $\mu_1,\ldots, \mu_{n}$ larger than $\lambda_0$ and in the resolvent set of $A$ such that
	\begin{align}
		\label{eq:thm:half-estimates:product-lower-estimate}
		\Res(\mu_k,A) \cdots \Res(\mu_1,A)f \succeq -u
	\end{align}
	for all $k \in \{0,\ldots,n\}$. Then Lemma~\ref{lem:resolvent-expansion-at-multiple-points} (applied to $m=n$) along with Remark~\ref{rem:dom-assumption} yields the required estimate.
	
	For the construction, first note that since $f \ge 0 \ge -u$, the case $k=0$ is trivial.
	Now, assume that $\mu_1, \ldots, \mu_k$ have already been constructed for some $k<n$ such that the estimate~\eqref{eq:thm:half-estimates:product-lower-estimate} holds. 
	Then, because the eigenvector $v$ from the spectral assumption satisfies $v\succeq u$, there exists a constant $c > 0$ such that 
	\[
		g := \Res(\mu_k,A) \ldots \Res(\mu_1,A)f + c v \ge 0.
	\]
	By assumption, this implies that there exists $\mu_{k+1}>\lambda_0$ and a constant $d_1>0$ such that $\Res(\mu_{k+1},A)g + d_1 u \ge 0$. 
	On the other hand, combining the domination
	and spectral assumptions, we have that $0\leq v\in\ker(\lambda_0-A)\subseteq \dom{A^n}\subseteq E_u$. 
	We can thus choose a constant $d_2>0$ such that $(\mu_{k+1}-\lambda_0)\Res(\mu_{k+1},A)v=v\leq d_2 u$.
	Therefore, 
	\begin{align*}
		0 
		& \le 
		\Res(\mu_{k+1},A) \Res(\mu_k,A) \ldots \Res(\mu_1,A)f + c\ \Res(\mu_{k+1},A) v + d_1 u 
		\\
		& \le 
		\Res(\mu_{k+1},A) \Res(\mu_k,A) \ldots \Res(\mu_1,A)f + \left(\frac{cd_2}{\mu_{k+1} - \lambda_0} + d_1\right) u;
	\end{align*}
	which readily yields
	\[
		\Res(\mu_{k+1},A) \Res(\mu_k,A) \ldots \Res(\mu_1,A)f \succeq -u.
		\hfill \qedhere
	\]
\end{proof}

Now we can prove our main result.

\begin{proof}[Proof of Theorem~\ref{thm:main}]
	``(iii) $\Rightarrow$ (i)'':
	This implication is a direct consequence of \cite[Theorem~4.4 and Corollary~3.3]{DanersGlueckKennedy2016b}. 
	
	``(i) $\Rightarrow$ (ii)'': If (i) holds, then so does the spectral assumption -- due to \cite[Theorem~4.1]{DanersGlueck2017} and \cite[Corollary~3.3]{DanersGlueckKennedy2016b} -- and in turn, (ii). 
	
	``(ii) $\Rightarrow$ (iii)'':
	We assume (ii), and we only have to show that $\dom A \subseteq E_u$.
	Due to the individual anti-maximum principle in assertion~(ii), Theorem~\ref{thm:half-estimates}(b) tells us that $\Res(\lambda,A)f \preceq u$ for all $\lambda > \lambda_0$ and all $f \in E_+$.
	By combining this with the individual maximum principal in assertion~(ii), we see that, for all $f \in E_+$, the estimate $0 \le \Res(\lambda,A)f \preceq u$ holds for all $\lambda$ in a $f$-dependent right neighbourhood of $\lambda_0$.
	
	Because the positive cone of a Banach lattice is generating, we've shown that, each $f \in E$ satisfies $\Res(\lambda,A)f \in E_u$ for all $\lambda$ in a $f$-dependent right neighbourhood of $\lambda_0$.
	We can now employ Corollary~\ref{cor:eventually-in-subspace-countable}(a) to conclude that $\dom A \subseteq E_u$.
\end{proof}

\begin{remark}
	\label{rem:spectral-assumption}
	Since both \cite[Theorem~4.1]{DanersGlueck2017} and \cite[Corollary~3.3]{DanersGlueckKennedy2016b} do not impose any domination assumptions,  
	such assumptions are not needed for the implication (i) $\Rightarrow$ (ii) in Theorem~\ref{thm:main} to be true.
\end{remark}

It is possible to give a proof of the implication ``(ii) $\Rightarrow$ (iii)'' which does not rely on the concept of operator ranges and Baire's theorem. 
A major drawback of the alternative proof is that it is technically more involved and, thus, less transparent. 
Moreover, it uses consequences of the resolvent equation in more than one place, due to which it is less likely to be adaptable to different situations.
However, for the benefit of the interested reader, we include the proof here.

\begin{proof}[Alternate proof of ``\,{\upshape (ii)} $\Rightarrow$ {\upshape (iii)}\!'' in Theorem~\ref{thm:main}]	
	Assume that (ii) is true. We want to prove that $\dom{A} \subseteq E_u$, for which it suffices to prove that $\Res(\mu,A)E \subseteq E_u$ for some $\mu\in \rho(A)$.

	Recall again that $n \ge 0$ denotes the integer that occurs in the domination assumption.
	If $n=0$, the conclusion is, of course, trivial. So, let $n\geq 1$ and fix $f\in E_+$.
	We assert that there exist $\mu_1,\ldots,\mu_n\in \rho(A)$ such that
	\begin{align}
		\label{thm:main:product-double-estimate}
		0\leq \Res(\mu_k,A) \ldots \Res(\mu_1,A)f \preceq u
	\end{align}
	for all $k\in \{1,\ldots,n\}$. 
	
	To see this, first note that since $f\geq 0$, condition~(ii) implies that for some $\lambda<\lambda_0$, we have $\Res(\lambda,A)f \leq 0\preceq u$. We can hence apply Theorem~\ref{thm:half-estimates}(b), and doing so in conjunction with~(ii), gives the existence of $\mu_1>\lambda_0$ such that $0\leq \Res(\mu_1,A)f \preceq u$. This establishes the assertion for the case $k=1$. 
	
	To proceed inductively, we now assume that for some $k<n$, there exist real numbers $\mu_1,\ldots,\mu_k\in \rho(A)$ satisfying the estimate~\eqref{thm:main:product-double-estimate}. 
	Then
	\[
		0\leq g:= \Res(\mu_k,A) \ldots \Res(\mu_1,A)f\in E.
	\] 
	Repeating the above argument for $g$ instead of $f$, we find $\mu_{k+1}>\lambda_0$ such that $0\leq \Res(\mu_{k+1},A)g\preceq u$, which completes the induction step. 
	
	In particular, it follows that $\prod_{j=1}^k \Res(\mu_j,A) f \in E_u$ for all $k\in \{1,\ldots,n\}$. Now, appealing to Lemma~\ref{lem:resolvent-expansion-at-multiple-points} together with Remark~\ref{rem:dom-assumption}, we obtain that $\Res(\mu,A)f \in E_u$ for each $\mu\in\rho(A)$.
	Since $E_+$ spans $E$, ergo, for each $\mu \in \rho(A)$ one has the inclusion $\Res(\mu,A) E \subseteq E_u$, as required.
\end{proof}

For eventually differentiable semigroups, \cite[Corollary~5.3]{DanersGlueckKennedy2016b} gives a characterization for individual eventual positivity of a semigroup whose generator satisfies the domination assumption from Setting~\ref{sett:main}. In this case, Theorem~\ref{thm:main} says that if the domination assumption does not hold for $n=1$, then the individual maximum and anti-maximum principle cannot both be true. This is also illustrated in Section~\ref{section:applications}.

\section{A sufficient condition for the domination assumption}
	\label{section:sufficient-condition-domination}

For sufficiently well-behaved concrete differential operators $A$, one way to check the domination assumption in Setting~\ref{sett:main} will typically be as follows: 
By employing an elliptic regularity result one gets that $\dom{A^n}$ is contained in a Sobolev space of an order that scales with $n$. So if $n$ is sufficiently large, one can employ a Sobolev embedding theorem (and combine it, if necessary, with the boundary conditions encoded in the domain of $A$) to obtain the embedding $\dom(A^n) \subseteq E_u$.
One problem with this approach is that it fails for operators for which elliptic regularity results are not available.

In fact, there are various differential operators for which elliptic regularity does not hold or is not known, but for which one has a parabolic estimate that yields $e^{tA}E \subseteq E_u$ for all $t > 0$, where $(e^{tA})_{t \in [0,\infty)}$ denotes the $C_0$-semigroup generated by $A$.
The following theorem (when applied to $V = E_u$) shows that a specific quantitative form of the property $e^{tA}E \subseteq E_u$ is equivalent to our domination assumption $\dom(A^n) \subseteq E_u$.

\begin{theorem}
	\label{thm:polynomial-smoothing}
	Let $E$ and $V$ be Banach spaces and assume that $V$ is continuously embedded in $E$.
	For a $C_0$-semigroup $(e^{tA})_{t \in [0,\infty)}$ on $E$, consider the following assertions:
	\begin{enumerate}[\upshape (i)]
		\item 
		There exist real numbers $t_0, q, c > 0$ such that, for all $t \in (0,t_0]$, 
		we have
		\begin{align*}
			\qquad 
			e^{tA}E \subseteq V
			\qquad \text{and} \qquad 
			\norm{e^{tA}}_{E \to V} \le c \, t^{-q}.
		\end{align*}
		
		\item 
		There exists an integer $n \ge 0$ such that $\dom(A^n) \subseteq V$.
	\end{enumerate}
	One always has {\upshape(i)} $\Rightarrow$ {\upshape(ii)}. 
	If the semigroup is analytic, then also {\upshape(ii)} $\Rightarrow$ {\upshape(i)}.
\end{theorem}

For the case $E=L^2$ and $V=L^\infty$, and under additional assumptions on the semigroup, implication ``(i) $\Rightarrow$ (ii)'' in Theorem~\ref{thm:polynomial-smoothing} can essentially be found in \cite[Theorem~2.4.1]{Davies1989}. 
The same argument can be adapted to the general case; though 
for the convenience of the reader, we prefer to include the details.

\begin{proof}[Proof of Theorem~\ref{thm:polynomial-smoothing}]
	``(i) $\Rightarrow$ (ii)'':
	Choose a real number $\lambda$ which is strictly larger than the growth bound of $(e^{tA})_{t \in [0,\infty)}$,
	and an integer $n \ge 1$ which satisfies $n > q$.
	Fix $f \in E$. Using
	\[
		\norm{e^{tA}f}_V \le \norm{e^{t_0 A}}_{E \to V} \norm{e^{(t-t_0)}f}_E
			\quad \text{for } t \in (t_0,\infty)
	\]
	along with the estimate in (i),
	we get that
	\begin{align*}
		\int_0^\infty t^{n-1} e^{-\lambda t} \norm{e^{tA}f}_V \dx t < \infty;
	\end{align*}
	the essential point is that the singularity of $t \mapsto \norm{e^{tA}f}_V$ at $t = 0$ gets suppressed since $n-1-q > -1$. (Observe that the mapping $(0,\infty) \ni t \mapsto e^{tA}f \in V$ is continuous, so the integrand is also continuous, and in turn, measurable.)
	Therefore, 
	\begin{align*}
		\int_0^\infty t^{n-1} e^{-\lambda t} e^{tA} f \dx t \in V,
	\end{align*}
	because the integral above exists as a Bochner integral in $V$.
	At the same time, the integral also exists as a Bochner integral in $E$, and this integral is equal to $(n-1)!\Res(\lambda,A)^n f$; see for instance \cite[Corollary~II.1.1]{EngelNagel2000}.
	
	The continuity of the embedding $V$ into $E$ implies that both Bochner integrals coincide, so $(n-1)!\Res(\lambda,A)^n f \in V$.
	In particular, $\Res(\lambda,A)^n f\in V$. 
	This shows that $\dom(A^n)\subseteq V$.

	``(ii) $\Rightarrow$ (i)'':	
	We now assume that the semigroup is analytic. 
	By replacing $A$ with $A-\lambda$ for a sufficiently large number $\lambda > 0$ if necessary, we may assume that $0$ is in the resolvent set of $A$.
	Now we endow the spaces $\dom(A^k)$ for $k=0,1,\ldots,n$ with the graph norms
	\begin{align*}
		\norm{f}_{\dom(A^k)} := \norm{A^k f}_E,
	\end{align*}
	which renders them Banach spaces.
	
	If $n = 0$, then $V = E$, so there is nothing to prove; 
	hence, let $n \ge 1$ for the rest of the proof.
	We show that
	\begin{align*}
		\norm{e^{tA}}_{E \to \dom(A^n)} \le \tilde c \, t^{-n}
	\end{align*}
	for a constant $\tilde c > 0$ and all $t \in (0,n]$.
	Since the inclusion $\dom(A^n) \hookrightarrow V$ is, by the closed graph theorem, continuous, 
	assertion (i) then follows using the analyticity of the semigroup with $q = n$ and $t_0 = n$.
	
	Since the semigroup is analytic, the number
	\begin{align*}
		M := \sup_{t \in (0,1]} \norm{t A e^{tA}}_{E \to E}
	\end{align*}
	is finite \cite[Theorem~II.4.6(a) and~(c)]{EngelNagel2000}
	(this reference is formulated for semigroups that are bounded on a sector, and thus has to be applied to a rescaled version of our semigroup).
	
	Consider an integer $k \in \{0, \ldots, n-1\}$ and a vector $f \in \dom(A^k)$. 
	Then, for $t \in (0,1]$, we have $e^{tA}f \in \dom(A^{k+1})$ and
	\begin{align*}
		\norm{e^{tA}f}_{\dom(A^{k+1})} = \norm{A e^{tA} A^k f}_E \le \frac{M}{t} \norm{f}_{\dom(A^k)},
	\end{align*}
	so $\norm{e^{tA}}_{\dom(A^k) \to \dom(A^{k+1})} \le M t^{-1}$ for $t \in (0,1]$.
	For $t \in (0,n]$ this implies
	\begin{align*}
		\norm{e^{tA}}_{E \to \dom(A^n)}
		\le 
		\prod_{k=0}^{n-1}
		\norm{e^{\frac{t}{n} A}}_{\dom(A^{k}) \to \dom(A^{k+1})}
		\le 
		(nM)^n \, t^{-n},
	\end{align*}
	which concludes the proof.
\end{proof}

\begin{remarks}\label{rem:polynomial-smoothing}
	(a) 
	Observe that the implication ``(ii) $\Rightarrow$ (i)'' in Theorem~\ref{thm:polynomial-smoothing} is, in general, false without the assumption of analyticity. 
	Just consider a $C_0$-group with unbounded generator $A$ and define $V := \dom A$ to obtain a counterexample where one does not even have $e^{tA}E \subseteq V$ for any time $t$.

	(b)
	The proof of Theorem~\ref{thm:polynomial-smoothing} shows that,
	in the implication ``(i) $\Rightarrow$ (ii)'', one can choose $n$ to be the smallest integer which satisfies $n > q$ 
	(and hence, one can choose $n = q+1$ if $q$ is an integer).		
	For the converse implication, ``(ii) $\Rightarrow$ (i)'', one can choose $q$ to be equal to $n$.
	
	(c)
	Since the condition $\dom{(A^n)} \subseteq V$ is equivalent to $\Res(\lambda,A)^nE \subseteq V$ for one (equivalently: all) $\lambda$ in the resolvent set of $A$, it follows from Theorem~\ref{thm:countable-individual-range-condition} that condition~(ii) in Theorem~\ref{thm:polynomial-smoothing} is equivalent to the condition that, for some $\lambda$ in the resolvent set, the following holds:  For each $f \in E$, there exists an ($f$-dependent) integer $n \ge 0$ such that $\Res(\lambda,A)^n f\in V$.
\end{remarks}

\section{Examples and applications}
	\label{section:applications}

In this section, we consider a couple of concrete differential operators to demonstrate how Theorem~\ref{thm:main} can be applied to characterize when they satisfy the individual anti-maximum principle.

\subsection{The Laplacian with Robin boundary conditions}
	\label{subsection:laplacian-robin}

Let $\Omega$ be a bounded domain in $\bbR^d$; for the sake of simplicity we assume that $\Omega$ has $C^\infty$-boundary, and we let $\beta: \partial \Omega \to \bbR$ be a $C^1$-function. Fix $p \in (1,\infty)$. 
We consider the Robin Laplace operator $\Delta_\beta: L^p(\Omega) \supseteq \dom(\Delta_\beta) \to L^p(\Omega)$ given by
\begin{align*}
	\begin{split}
		\dom(\Delta_\beta) & = \left\{ u \in W^{2,p}(\Omega): \frac{\partial}{\partial \nu} u = \beta u \text{ on } \partial \Omega\right\}, \\ 
		\Delta_\beta u     & = \Delta u;
	\end{split}
\end{align*}
where $\frac{\partial}{\partial \nu} u$ denotes the outer normal derivative.

The operator $\Delta_\beta$ generates an analytic $C_0$-semigroup on $L^p(\Omega)$ \cite[Theorem~5.6 on p.\,189]{Tanabe1997} and it can be shown -- for instance, by applying form methods to the case $p=2$ and then extrapolating those properties to other $L^p$-spaces -- that this semigroup is both positive and irreducible. 

Let $v$ denote the eigenvector of $\Delta_\beta$ associated to the eigenvalue $\spb(A)$. Then $v$ is an element of $C(\overline{\Omega})$ and it satisfies $v\succeq \one$, where $\one$ denotes the constant one function. 
This is true even for more general elliptic operators and much rougher domains and can, for instance, be found in \cite[Theorem~4.5(b)]{ArendtterElstGlueck2020}.
Also by the positivity and the irreducibility, the eigenspace of the dual operator $A'$ corresponding to the eigenvalue $\spb(A)$ contains a strictly positive functional \cite[Proposition~C-III-3.5]{Nagel1986}.

We also note that, for $p=2$, the associated sesquilinear form has form domain $H^1(\Omega)$, which embeds continuously into $L^q(\Omega)$ for some $q>2$, so using an ultracontractivity argument (see \cite[Theorem in Section~7.3.2]{Arendt2004}), we obtain that
$
	e^{t\Delta_\beta}L^p(\Omega) \subseteq L^\infty(\Omega) = L^p(\Omega)_{\one}
$
and, even more, that the growth condition in Theorem~\ref{thm:polynomial-smoothing}(i) is satisfied. As a result, the aforementioned theorem implies that there exists $n\in\bbN$ such that $\dom\big((\Delta_\beta)^n\big)\subseteq L^p(\Omega)_{\one}$. We have thus shown that both the domination and the spectral assumptions in Setting~\ref{sett:main} are satisfied for $\lambda_0 = \spb(\Delta_\beta)$ and $u = \one$.

We also note that the resolvent of $\Delta_\beta$ is positive on the right of $\spb(\Delta_\beta)$ since the semigroup $(e^{t\Delta_\beta})_{t \in [0,\infty)}$ is positive. In particular, in the terminology of Theorem~\ref{thm:main}, the individual (and in fact, even the uniform) maximum principle is satisfied at $\lambda_0 = \spb(\Delta_\beta)$. 
Let us now discuss under which conditions the individual anti-maximum principle is also satisfied.

\begin{theorem}
	\label{thm:robin-laplace}
	The Robin Laplace operator $\Delta_\beta$ satisfies the individual anti-maximum principle at $\lambda_0 = \spb(\Delta_\beta)$ if and only if $p > d/2$.
\end{theorem}
\begin{proof}
	Since the domination and the spectral assumption are satisfied for $u=\one$ and since $\Delta_\beta$ satisfies the individual maximum principle, it follows from Theorem~\ref{thm:main} that the individual anti-maximum principle is satisfied if and only if
	\[
		\dom(\Delta_\beta) \subseteq L^2(\Omega)_{\one} = L^\infty(\Omega).
	\]
	
	If $p > d/2$, then by the Sobolev embedding theorem, $W^{2,p}(\Omega)$ embeds into $L^\infty(\Omega)$ and hence, indeed, $\dom(\Delta_\beta) \subseteq L^\infty(\Omega)$.
	
	If, on the other hand, $p \le d/2$, then there exists a function in $W^{2,p}(\Omega)$ which has compact support within $\Omega$ (and is thus an element of $\dom(\Delta_\beta)$), but is not contained in $L^\infty(\Omega)$ \cite[Examples~5.25 and~5.26]{Adams1975}. 
	Wherefore, $\dom(\Delta_\beta) \not\subseteq L^\infty(\Omega)$ in this case.
\end{proof}

A generalisation of this result will (under slightly more restrictive assumptions on $\beta$) be discussed in Section~\ref{subsection:powers-of-robin-laplacian}.

\subsection{On the Laplacian with Dirichlet boundary conditions} 

For the Laplace operator with Dirichlet boundary conditions $\Delta_{\operatorname{Dir}}$, it is also possible to characterize whether an anti-maximum principle at $\spb(\Delta_{\operatorname{Dir}})$ holds.
This was done by Sweers \cite{Sweers1997} who proved that, on a smooth bounded domain in $\bbR^d$, the individual anti-maximum principle for the Dirichlet Laplacian on $L^p$ holds if and only if $p > d$.
The crucial point here is the boundary behaviour of functions lying in the domain of the operator, which explains that $p$ needs to be larger (when compared to the Robin Laplacian) for the individual anti-maximum principle to hold.

Sweers' argument can also be rephrased to fit into the framework of our Theorem~\ref{thm:main} -- however, our theorem does not appear to significantly simplify the argument, so we refrain from discussing the details here.

We point out that for $d\ge 2$, the individual anti-maximum principle is not uniform; see \cite[Proposition~6.1(a)]{AroraGlueck2021b} and the introductions of \cite{Sweers1997,ClementPeletier1979}.

An important observation here is that while the individual anti-maximum principle holds only for some $p$, its uniform counterpart is not true for any $p$.
In fact, this characteristic is not unique to the Dirichlet Laplacian. The uniform (anti-)maximum principle is an operator inequality; hence, for a differential operator which acts on the entire $L^p$-scale, if the principle holds for some $p$, then by a simple density argument, it automatically holds for each $p$. The individual (anti-)maximum principle, on the other hand, is more subtle and its validity may rely on the choice of $p$.

\subsection{A system of Neumann Laplacians coupled by a matrix-valued potential}
\label{subsection:neumann-laplacian-coupled}

We consider a situation similar to Subsection~\ref{subsection:laplacian-robin} -- however, we couple multiple Laplace operators by a matrix-valued potential now. 
For the sake of simplicity, we consider Neumann boundary conditions only.

Let $\Omega \subseteq \bbR^d$ be a bounded domain with $C^\infty$-boundary, let $N \ge 1$ be an integer and let $V: \Omega \to \bbR^{N \times N}$ be a measurable and bounded function. 
We assume that, for each $x \in \Omega$, all off-diagonal entries of the matrix $V(x)$ are non-negative and that the matrix that is obtained from $V(x)$ by setting all diagonal entries to $0$ is irreducible. 
On the space $L^p(\Omega; \bbC^N)$, we consider the operator $A$ given by
\begin{align*}
	D(A) & = D(\Delta_{\operatorname{Neu}}) \times \ldots \times D(\Delta_{\operatorname{Neu}}), \\
	Aw   & =
	\begin{pmatrix}
		\Delta w_1 \\ \vdots \\ \Delta w_N
	\end{pmatrix}
	+ 
	V w;
\end{align*}
where $D(\Delta_{\operatorname{Neu}}) = \{h \in W^{2,p}(\Omega;\bbC) \; \frac{\partial}{\partial \nu} h = 0 \text{ on } \partial \Omega \}$ denotes the domain of the Neumann Laplace operator on $L^p(\Omega)$.
Since $V$ acts as a bounded linear operator on $L^p(\Omega; \bbC)$ and the semigroup generated by this operator is positive, it follows that the semigroup generated by $A$ is also positive;
in fact, it is even irreducible, as follows, for example, from \cite[Proposition~C-III-3.3]{Nagel1986}.

Because $A$ has compact resolvent, we can conclude that $\spb(A)$ is an algebraically simple eigenvalue of $A$ and that the corresponding eigenspace is spanned by a vector $v$ which is a quasi-interior point of the positive cone of $L^p(\Omega;\bbC^N)$; moreover, the dual eigenspace is also spanned by a strictly positive functional. This follows from \cite[Proposition~C-III-3.5]{Nagel1986}.

That fact that $v$ is a quasi-interior point means that each of the components $v_1, \ldots, v_N \in L^p(\Omega;\bbC)$ is $> 0$ almost everywhere on $\Omega$. 
In fact, we even have $v_1, \ldots, v_n \succeq \one$. 
To see this, choose a number $c \in [0,\infty)$ sufficiently large such that $V(x)+c \id_{\bbC^N} \ge 0$ for each $x \in \Omega$. 
Then $(e^{tc}e^{tA})_{t \in [0,\infty)}$ dominates the semigroup generated by
\begin{align*}
	\begin{pmatrix}
		\Delta_{\operatorname{Neu}} &        &                           \\ 
		                            & \ddots &                           \\
		                            &        & \Delta_{\operatorname{Neu}}
	\end{pmatrix}
\end{align*}
Hence, we obtain for the $k$-th component of the vector $v$, that
\begin{align*}
	e^{tc} e^{t\spb(A)} v_k = (e^{tc} e^{tA} v)_k \ge e^{t\Delta_{\operatorname{Neu}}} v_k \succeq \one,
\end{align*}
and thus $v_k \succeq \one$.
So the spectral assumption in Setting~\ref{sett:main} is satisfied for $\lambda_0 = \spb(A)$ and $u := (\one, \ldots, \one)$.
On the other hand, the domination assumption in Setting~\ref{sett:main} is satisfied for this $u$, as well. 
This follows from $\big(L^p(\Omega;\bbC^N)\big)_u = L^\infty(\Omega;\bbC^N)$ along with Theorem~\ref{thm:polynomial-smoothing} and an ultracontractivity argument (see for instance \cite[Theorem in Section~7.3.2]{Arendt2004}).

Since the semigroup generated by $A$ is positive, the resolvent of $A$ is positive on the right of $A$, i.e., the individual (and in fact, even uniform) maximum principle holds at $\lambda_0 = \spb(A)$. 
For the individual anti-maximum principle, we now get the same result as in Theorem~\ref{thm:robin-laplace}:

\begin{theorem}
	The coupled Neumann Laplace operator $A$ satisfies the individual anti-maximum principle at $\lambda_0 = \spb(A)$ if and only if $p > d/2$.
\end{theorem}
\begin{proof}
	Given the facts discussed before the theorem, the proof is now the same as for Theorem~\ref{thm:robin-laplace}.
\end{proof}

\subsection{Powers of the Robin Laplacian}
	\label{subsection:powers-of-robin-laplacian}

We consider the same situation as in Section~\ref{subsection:laplacian-robin}, but now we assume in addition that $\beta \succeq \one$ and, for the sake of simplicity, that $\beta$ is a $C^\infty$-function.
Since $\beta \succeq \one$, the spectral bound $\spb(\Delta_\beta)$ of the Robin Laplacian $\Delta_\beta$ is strictly negative.

Now, we fix an integer $k \ge 1$ and consider the operator $B := -(-\Delta_\beta)^k$ on $L^p(\Omega)$. 
The semigroup generated by $B$ is not positive unless $k=1$ (however, we note that eventual positivity of the semigroup generated by the square of the Robin Laplacian on the space $C(\overline{\Omega})$ was discussed in \cite[Section~6.4]{DanersGlueckKennedy2016a}).

By the spectral mapping theorem, the spectral bound of $B$ also satisfies $\spb(B) < 0$, and for the resolvent at the point $0$ we have
\begin{align*}
	\Res(0,B) = \Res(0,\Delta_\beta)^k \ge 0.
\end{align*}
Hence, it follows from the Taylor series expansion of the resolvent that $\Res(\lambda,B) \ge 0$ for all $\lambda \in (\spb(B),0]$ (see \cite[Proposition~3.2(i)]{DanersGlueck2018a} and \cite[Proposition~4.2(i)]{DanersGlueckKennedy2016a}).

In particular, the operator $B$ satisfies the individual (even the uniform) maximum principle at $\lambda_0 = \spb(B)$.
As pointed out in Subsection~\ref{subsection:laplacian-robin} the operator $\Delta_\beta$ satisfies the domination assumption in Setting~\ref{sett:main} for $u = \one$ and hence, so does $B$.
Moreover, since $\spb(\Delta_\beta) < 0$, the eigenspace of $B$ for the eigenvalue $\spb(B)$ coincides with the eigenspace of $\Delta_\beta$ for the eigenvalue $\spb(\Delta_\beta)$, and the same is true for the dual operators. 
Thus, $B$ also satisfies the spectral assumption from Setting~\ref{sett:main} (since the same is true for $\Delta_\beta$).

Let us now analyse whether $B$ satisfies the individual anti-maximum principle at $\spb(B)$:

\begin{theorem}
	The operator $B$ satisfies the individual anti-maximum principle at $\lambda_0 = \spb(B)$ if and only if $kp > d/2$.
\end{theorem}
\begin{proof}
	Since both $\partial \Omega$ and $\beta$ are smooth, it follows from elliptic regularity that $\dom(B) \subseteq W^{2k,p}(\Omega)$.
	So if $kp > d/2$, then by the Sobolev embedding theorem, $\dom(B) \subseteq L^\infty(\Omega) = (L^p(\Omega))_{\one}$, and consequently Theorem~\ref{thm:main} implies that $B$ satisfies the individual anti-maximum principle at $\spb(B)$.
	
	Conversely, if $kp \le d/2$, then we can find a function $w: \Omega \to \bbR$ which has compact support wihtin $\Omega$ and is in $W^{2k,p}(\Omega)$, but not in $L^\infty(\Omega)$. 
	So, this function $w$ is an element of $\dom(B)$, but not of $L^\infty(\Omega) = (L^p(\Omega))_{\one}$.
	Whence, $B$ does not satisfy the anti-maximum principle at $\spb(B)$ according to Theorem~\ref{thm:main}.
\end{proof}

\subsection{The Dirichlet-to-Neumann operator}

Let $\Omega \subseteq \bbR^d$ be a bounded domain with $C^{\infty}$-boundary. 
Let $\Delta_{\operatorname{Dir}}$ denote the Dirichlet Laplace operator  on $L^2(\Omega)$ and consider a function $V\in L^\infty(\Omega,\bbR)$ such that the infinity norm of the negative part of $V$, $\norm{V^-}_\infty$, is strictly smaller than the smallest eigenvalue of $-\Delta_{\operatorname{Dir}}$. In this situation, the operator $-\Delta_{\operatorname{Dir}} + V$ is positive definite; in particular, $0$ is not in its spectrum. 
This enables us, for every $\phi \in H^{1/2}(\partial \Omega)$, to uniquely solve
\begin{align*}
	(-\Delta_{\operatorname{Dir}} + V) u &= 0 \quad \text{weakly on } \Omega,\\ 
	\trace u&=\phi;
\end{align*} 
where $u\in W^{1,2}(\Omega)$. 
Finally, denote the weak outer normal derivative of $u$ on the boundary $\partial \Omega$ by $\partial_{\nu} u$ whenever it exists (see \cite{terElstOuhabaz2019a} or \cite{terElstOuhabaz2019b} for a definition of weak normal derivative). Then the operator
$
	\phi\mapsto \partial_\nu u
$
is called the \emph{Dirichlet-to-Neumann} operator. We denote this operator by $D_V$ and its domain by $\dom{D_V}$. For an exact definition, and more details about the Dirichlet-to-Neumann operator, we refer the reader to \cite{terElstOuhabaz2019a} or \cite{terElstOuhabaz2019b}. 
This operator is self-adjoint on $L^2(\partial \Omega)$ with compact resolvent and $-D_V$ generates a holomorphic $C_0$-semigroup on $L^2(\partial \Omega)$ of angle $\frac{\pi}2$. 
Moreover, the semigroup generated by $-D_V$ is positive (see \cite[Theorem~2.3]{terElstOuhabaz2014} or \cite[Theorem~1.1(a)]{AbreuCapelato2018}) and even irreducible, as shown in \cite[Theorem~1.1(c)]{AbreuCapelato2018} (note that there is a small inaccuracy in the statement of this result: the estimate that is assumed there should actually be the stronger estimate that is also assumed in \cite[Theorem~1.1(a)]{AbreuCapelato2018}).

Due to the positivity of the semigroup, the resolvent satisfies $\Res(\mu, -D_V) \geq 0$ for all $\mu >\spb(-D_V)$. In particular, we have an individual maximum principle at $\lambda_0=\spb(-D_V)$. We now show that for $d \ge 3$, there is no anti-maximum principle. 

\begin{theorem}\label{thm:anti-max-DtN}
	The individual anti-maximum principle holds at $\lambda_0=\spb(-D_V)$ if and only if $d \leq 2$.
\end{theorem}
\begin{proof}
	According to \cite[Theorem~2.3]{terElstOuhabaz2019a}, there exist $c,q>0$ such that
	\begin{align*}
		\qquad 
		e^{-t D_V}L^2(\partial \Omega)\subseteq L^\infty(\partial \Omega)
		\qquad \text{and} \qquad 
		\norm{e^{-tD_V}}_{L^2(\partial \Omega)\to L^\infty(\partial \Omega)} \le c \, t^{-q}.
	\end{align*}
	for all $t\in (0,1]$. Therefore, by virtue of Theorem~\ref{thm:polynomial-smoothing}, there exists an integer $n\geq 0$ such that $\dom(D_V^n)\subseteq L^\infty(\partial \Omega)$. Combining this with $L^\infty(\partial \Omega) = L^2(\partial \Omega)_{\one}$ -- where $\one$ denotes the constant function with value one -- we obtain that the domination assumption in Setting~\ref{sett:main} is fulfilled for $u = \one$. 
	
	Moreover, since the semigroup generated by $-D_V$ is both positive and irreducible, the spectral bound $\spb(-D_V)$ is an eigenvector of $-D_V$ with corresponding eigenspace spanned by a vector $v$ which is strictly positive almost everywhere; see \cite[Proposition~C-III-3.5]{Nagel1986}. Let us show that even $v \succeq \one$, so that also the spectral assumption in Setting~\ref{sett:main} is satisfied for $u = \one$ and $\lambda_0 = \spb(-D_V)$.
	To this end, we note that it follows from $e^{-t D_V}L^2(\partial \Omega)\subseteq L^\infty(\partial \Omega)$ and from \cite[Lemma~2.2]{terElstOuhabaz2019b} that $e^{-t D_V}$ maps $L^2(\partial \Omega)$ into $C(\partial \Omega)$ for each $t > 0$; moreover, this mapping is continuous by the closed graph theorem.
	Now, we use an argument inspired by \cite[Section~3]{ArendtterElstGlueck2020}: 
	It was shown in \cite[Theorem~1.1]{terElstOuhabaz2019b} that $\left(e^{-t D_V} \restrict{C(\partial \Omega)}\right)_{t\geq 0}$ is a $C_0$-semigroup on $C(\partial \Omega)$. Hence, there exists a time $t > 0$ such that $e^{-t D_V} \one \succeq \one$, i.e., such that $e^{-t D_V}$ maps a positive element of $L^2(\partial \Omega)$ to a quasi-interior point of $C(\partial \Omega)$. 
	Such an operator automatically maps all quasi-interior points of $L^2(\partial \Omega)$ to quasi-interior points of $C(\partial \Omega)$ (see \cite[Proposition~2.21]{GlueckWeber2020} and note that the notion \emph{almost interior point} that occurs there is, within Banach lattices, equivalent to \emph{quasi-interior point} \cite[pp.\,242--243]{GlueckWeber2020}).
	As a result, $v = e^{-tD_V}v$ is a quasi-interior point of $C(\partial \Omega)$ and thus satisfies $v \succeq \one$.
	
	As mentioned above, the positivity of the semigroup implies that the (uniform and hence, in particular, individual) maximum principle holds at $\lambda_0 = \spb(-D_V)$.  Lastly, we know that
	$
		\dom(-D_V)= H^{1}(\partial\Omega);
	$
	this follows from \cite[Theorem~5.2]{BehrndtterElst2015} since $0 \in \rho(-\Delta_{\operatorname{Dir}}+V)$.
	Using Sobolev embedding theorems on manifolds (see, for instance, \cite[Theorem~6.2]{HebeyRobert2008}), we obtain
	\[
		\dom(-D_V)\subseteq C(\partial\Omega) \subseteq L^\infty(\partial\Omega)= L^2(\partial \Omega)_{\one}
	\]
	 if  $d<3$. On the other hand, if $d\geq 3$, 
	 then $\partial \Omega$ is a manifold of dimension at least two and hence, $H^{1}(\partial \Omega)\not\subseteq L^\infty(\partial\Omega)$ (see \cite[Examples~5.25 and~5.26]{Adams1975}, \cite[Proposition~4.1]{HebeyRobert2008}, or \cite[Example~2.5.1]{Kesavan1989}).
	  The result thus follows from Theorem~\ref{thm:main}.
\end{proof}

\subsection*{Acknowledgements} 

The authors would like to express their gratitude to Mikhail Ostrovskii \cite{Ostrovskii2019} and Jochen Wengenroth \cite{Wengenroth2021} for two answers on MathOverflow that contain useful references about operator ranges.
The question of whether assertion~(b) in Corollary~\ref{cor:eventually-in-subspace-countable} holds was kindly brought to the authors' attention by Bálint Farkas. 

Part of the work on this article was done when the first author visited the second author at Universität Passau, Germany.
	
The first named author was supported by 
Deutscher Aka\-de\-mi\-scher Aus\-tausch\-dienst (Forschungs\-stipendium-Promotion in Deutschland).

\bibliographystyle{plainurl}
\bibliography{literature}

\end{document}